\documentclass[12pt]{amsart}
\usepackage{amssymb,amsthm,amsmath,amsfonts}
\usepackage{geometry}
\usepackage{color}
\usepackage{amssymb}
\usepackage{verbatim}
\usepackage{amsmath}
\usepackage[T1]{fontenc}	
\usepackage[utf8]{inputenc}
\usepackage{enumerate}
\usepackage{amsfonts}
\usepackage{graphicx}
\usepackage{color}

\newcommand{\mb}[1]{\mathbb{{#1}}}
\newcommand{\mc}[1]{\mathcal{{#1}}}

\newcommand{\vp}{\varphi}
\newcommand{\ud}{\mathrm{d}}

\newcommand{\p}[1]{\mb{P} \left( #1 \right)}

\newcommand{\dd}{\mathrm{d}}
\newcommand{\E}{\mathbb{E}}
\newcommand{\1}{\textbf{1}}
\newcommand{\R}{\mathbb{R}}

\newcommand{\Z}{\mathbb{Z}}

\newcommand{\commnt}[1]{%
}%
\newcommand{\ve}{\varepsilon}

\DeclareMathOperator{\sgn}{sgn}

\DeclareMathOperator{\rk}{rank}

\theoremstyle{definition}
\newtheorem{theorem}{Theorem} 
\newtheorem{cor}{Corollary} 
\newtheorem{proposition}[theorem]{Proposition}

\newtheorem{lemma}[theorem]{Lemma}

\theoremstyle{remark}
\newtheorem*{remark}{Remark}

\begin{document}

	\newgeometry{tmargin=2.0cm, bmargin=2.0cm, lmargin=2.0cm, rmargin=2.0cm}
	

	\title{On the optimal $L_p$-$L_4$ Khintchine inequality }

	\author{Adam Barański}
	
	\address{(AB) University of Warsaw, Banacha 2, 02-097 Warsaw, Poland.}
	\email{ab429277@students.mimuw.edu.pl}
	
	\author{Daniel Murawski}
	
	\address{(DM) University of Warsaw, Banacha 2, 02-097 Warsaw, Poland.}
	\email{dk.murawski@student.uw.edu.pl}
	
	\author{Piotr Nayar}
	
	\address{(PN) University of Warsaw, Banacha 2, 02-097 Warsaw, Poland.}
	\email{nayar@mimuw.edu.pl}
	
	\author{Krzysztof Oleszkiewicz}
	\thanks{Research of the fourth named author was
		partially funded by the Deutsche Forschungsgemeinschaft (DFG, German Research Foundation) under Germany's Excellence Strategy – EXC-2047/1 –
		390685813. He thanks the Hausdorff Research Institute for Mathematics for its kind hospitality.}
	
	\address{(KO) University of Warsaw, Banacha 2, 02-097 Warsaw, Poland.}
	\email{koles@mimuw.edu.pl}

	\maketitle

	\begin{abstract}
		We derive  optimal dimension independent constants in the classical Khintchine inequality between the $p$th and fourth moment for $p\ge 4$. As an application we deduce stability estimates for the Khintchine inequality between the $p$th and second moment for $p \geq 4$.
	\end{abstract}
	
	\section{Introduction}
		
	Let $\ve_1, \ldots, \ve_n$ be i.i.d. symmetric Bernoulli random variables, that is, $\p{\ve_i=\pm 1}= \frac12$. For a real vector $a=(a_1,\ldots, a_n)$ we consider the random variable $S_a = \sum_{i=1}^n a_i \ve_i$, sometimes to be denoted by $S$ for simplicity.  The $p$th moment of $S$ is defined as  $\|S\|_p = (\E|S|^p)^{1/p}$, where $p>0$. Let $C_{p,q}$ be the best constant, independent of $a$ and $n$, in the  classical Khintchine inequality
	$
	\|S\|_p \leq C_{p,q} \|S\|_q.
	$

	Let us now briefly discuss the state of the art for the above  inequality. The inequality was first considered a century ago in \cite{K23} by Khintchine in his study of the law of the iterated logarithm and independently by Littlewood \cite{L30} in 1930. Khintchine showed  that the numbers $C_{p,q}$ are finite. By monotonicity of moments we easily see that $C_{p,q}=1$ for $p \leq q$. The best constants $C_{p,q}$ are known when one of the numbers $p,q$ equals $2$, in which case one of the sides of the inequality has a simple form. The optimal constant $C_{p,2}$ for $p>2$ equals $\gamma_p/\gamma_2$, where $\gamma_p=(\E|G|^p)^{1/p}$ for $G \sim \mc{N}(0,1)$. The equality holds asymptotically  when $a_i=n^{-1/2}$ and $n \to \infty$. For the constant $C_{2,q}$ with $q \in (0,2)$ a phase transition occurs, namely there is $q_0 \in (1,2)$  such that for $q_0 \leq q \leq 2$ one has $C_{2,q}=\gamma_2/\gamma_q$, whereas for $0<q \leq q_0$ we have $C_{2,q}=2^{\frac1q - \frac12}$ with equality  for $n=2$ and $a_1=a_2=1$. In fact $q_0$ is the solution to the equation $\Gamma(\frac{q+1}{2})=\frac{\sqrt{\pi}}{2}$, $q_0 \approx 1.84742$. The constant $C_{p,2}$ for $p$ even was found by Khintchine himself in \cite{K23}, whereas the constant $C_{p,2}$ for $p \geq 3$ was established by Whittle in \cite{W60} and independently by Young in \cite{Y76} who was not aware of Whittle's work. Szarek in \cite{S76} showed that  $C_{2,1}=\sqrt{2}$, answering the question of Littlewood from \cite{L30}. The remaining constants $C_{p,2}$ for $p \in (2,3)$ and $C_{2,q}$ for $q \in (0,2) \setminus \{1\}$ were found by Haagerup in his celebrated work \cite{H81} using Fourier methods, see also the article \cite{NP00} of Nazarov and Podkorytov for a simpler proof for $C_{2,q}$ and the article \cite{M17} of Mordhorst for a simpler proof in the case $C_{p,2}$ with $p \in (2,3)$, based on the idea of Nazarov and Podkorytov. In the case of even $p,q$ with $p$ divisible by $q$ the best constants were obtained by Czerwiński in \cite{C08}. In \cite{NO12} the optimal constants $C_{p,q}=\gamma_p / \gamma_q$ for all even numbers $p>q>0$ were found, see also a recent work \cite{HNT23} for an alternative proof. 
	
	One can also introduce the dimension dependent optimal constant $C_{p,q,n}$. By homogeneity  
	\[
	C_{p,q,n} = \max_{a \in S^{n-1}} \ \frac{\|S_a\|_p}{\|S_a\|_q}
	\] 
	and the maximum is achieved for some vector $a^\ast$ by Weierstrass extreme value theorem. A priori, the maximizer $a^\ast$ might not be unique. Plainly
	$C_{p,q} = \sup_{n \geq 1} \ C_{p,q,n}$. Clearly the sequence $C_{p,q,n}$ is non-decreasing, since  for every $a \in \R^n$ we have $S_a = S_{a'}$, where $a'=(a,0) \in \R^{n+1}$. The constants $C_{p,q}$ might or might not be achieved for some finite $n$. In the former case the sequence $(C_{p,q,n})_n$ becomes eventually constant and in the latter we in fact get a strict inequality $\|S\|_p < C_{p,q} \|S\|_q$.
	
	Let us notice that by taking $a=\frac{1}{\sqrt{n}}(1,\ldots, 1)$ with $n \to \infty$ and using central limit theorem one gets $C_{p,q} \geq \gamma_p / \gamma_q$, where $\gamma_p$ is the $p$th moment of a standard Gaussian $\mc{N}(0,1)$ random variable $G$.  The main result of this article reads as follows.
	
	\begin{theorem}\label{thm:cp4}
		For $p \geq 4$ we have $C_{p,4}= \gamma_p/\gamma_4$.
	\end{theorem}
	
	\noindent In order to prove the above theorem we provide two separate arguments in two different regimes of $p$. The first approach works in the case $p \in [4,8]$, whereas  the second technique is applied in the case $p \geq 5$ and relies on the following proposition.
	
	\begin{proposition} \label{prop:x}
		Let $S_n = \ve_1+\ldots + \ve_n$. We have 
		\[
		C_{p,4,n+1} = \sup_{x \geq 1} \ \frac{\|x+S_n\|_p}{\|x+S_n\|_4}.
		\]
	\end{proposition}
	
	\noindent We conjecture that in fact the above supremum is attained for $x=1$.	
	
	\noindent Theorem \ref{thm:cp4} gives a stability result for the dimension independent classical Khintchine inequality. A more general result was independently obtained in a recent work \cite{J25}, see Theorem 1 therein. 
	
	\begin{cor}
		For any $a \in S^{n-1}$ and $p \geq 4$ we have 
		\[
		\frac{\|S\|_p}{\|G\|_p} \leq 1 - \frac16 \sum_{i=1}^n a_i^4.
		\]
	\end{cor}
	
	\noindent Indeed, let us recall that 
	\[
	\E S_a^4 = 3\left(\sum_{i=1}^n a_i^2\right)^2 - 2\sum_{i = 1}^n a_i^4
	\]
	and $\E G^4 = 3$ which gives
	\[
	\frac{\|S\|_p}{\|G\|_p} \leq \frac{\|S\|_4}{\|G\|_4} = \left( 1- \frac23 \sum_{i=1}^n a_i^4 \right)^{\frac14} \leq 1 - \frac16 \sum_{i=1}^n a_i^4
	\]
	as $(1-x)^q \leq 1-qx$ for $x,q \in [0,1]$. Note that if $a_i=n^{-1/2}$ for $i=1,\ldots, n$ then the right hand side is $1-\frac{1}{6n}$, whereas the left hand side is lower bounded by $1-\frac{p}{2n}$, according to the following proposition.  
	\begin{proposition}\label{prop:Sn_to_G}
		For any positive integer $n$ and $p \geq 3$ we have
		$$\Big\|\frac{S_n}{\sqrt{n}}\Big\|_p \geq e^{-p/2n}\|G\|_p.$$
	\end{proposition}
	
	\noindent This shows optimality of our stability result in any fixed dimension, up to a constant depending only on $p$. 
	
	The article is organized as follows. In the next section we prove Theorem \ref{thm:cp4} in the case $p \in [4,8]$. In Section \ref{sec:p>5} we derive Theorem \ref{thm:cp4} in the case $p \geq 5$ from Proposition \ref{prop:x}. Section \ref{sec:prop} is devoted to the proof of Proposition \ref{prop:x}. Finally, in Section \ref{sec:opt} we prove Proposition \ref{prop:Sn_to_G}.

	\section{Proof of Theorem \ref{thm:cp4} for $4 \leq p \leq 8$}
	
	\noindent Let us first prove two lemmas.

	\begin{lemma}\label{lem:KO-1}
		For every ${\mathcal{C}}^{1}$ function $f: \R \to \R$ and for every Rademacher sum $S=\sum_{i=1}^{n} a_{i}\ve_{i}$, 
		\[
		\E[Sf(S)]=\sum_{i=1}^{n} a_{i}^{2}\,\E\big[f'(S_{i})\big],
		\]
		where 
		\[
		S_{i}=a_{1}\ve_{1}+\ldots+a_{i-1}\ve_{i-1}+a_{i}U+a_{i+1}\ve_{i+1}+\ldots+a_{n}\ve_{n}
		\] 
		for $1 \leq i \leq n$ and $U$ is a random variable distributed uniformly on $[-1,1]$ and independent of the sequence $\ve_{1}$, $\ve_{2}, \ldots, \ve_{n}$. In particular, for $p\geq 2$,  we obtain
		\[
		\E|S|^{p}=(p-1) \cdot \sum_{i=1}^{n} a_{i}^{2}\,\E|S_{i}|^{p-2}.
		\]
	\end{lemma}

	\begin{proof}
		We have
		\[
		\E[Sf(S)]=\sum_{i=1}^n a_{i}\,\E[\ve_{i}f(S)]
		\]
		and
		\begin{align*}
			\E[\ve_{i}f(S)]& =\frac{1}{2}\,\E\big[f(a_{1}\ve_{1}+\ldots+a_{i-1}\ve_{i-1}+a_{i}+a_{i+1}\ve_{i+1}+\ldots+a_{n}\ve_{n}) \\
			& \qquad \qquad -
			f(a_{1}\ve_{1}+\ldots+a_{i-1}\ve_{i-1}-a_{i}+a_{i+1}\ve_{i+1}+\ldots+a_{n}\ve_{n})\big] \\
			& =\frac{1}{2}\,\E\int_{-a_{i}}^{a_{i}} f'(a_{1}\ve_{1}+\ldots+a_{i-1}\ve_{i-1}+t+a_{i+1}\ve_{i+1}+\ldots)\,\ud t
			=a_{i} \cdot \E\big[f'(S_{i})\big].
		\end{align*}
		Taking $f(x)= |x|^{p-1} \sgn(x)$, so that $xf(x)=|x|^{p}$ and $f'(x)=(p-1)|x|^{p-2}$, gives the second part.
	\end{proof}

	\begin{lemma}\label{lem:KO-2}
		Let $a_{1}, a_{2}, \ldots, a_{n}$ be such that $\sum_{i=1}^{n} a_{i}^{2}=1$. Then
		\[
		\sum_{i=1}^{n} a_{i}^{2}\Big(1-\frac{2}{3}a_{i}^{2}\Big)^{3} \leq \Big(1-\frac{2}{3}\sum_{i=1}^{n} a_{i}^{4}\Big)^{2}.
		\]
	\end{lemma}
	
	\begin{proof}
		For  $\alpha>0$ let us denote $\sum_{i=1}^{n} a_{i}^{\alpha}$ by $s_{\alpha}$. Since $x^6 \leq \frac12 (x^4+x^8)$ for any real number $x$, we get $s_6 \leq \frac12(s_4+s_8)$. Clearly we also have $s_8 \leq s_4^2$.  Therefore
		\begin{align*}
			\sum_{i=1}^{n} a_{i}^{2}\Big(1-\frac{2}{3}a_{i}^{2}\Big)^{3} &  = 1-2s_4 + \frac43 s_6 - \frac{8}{27}s_8 \leq 1-2s_4 + \frac23 (s_4+s_8) - \frac{8}{27}s_8  \\
			& \leq 1 - \frac43 s_4 + \frac{10}{27} s_8 \leq 1 - \frac43 s_4 + \frac{10}{27} s_4^2 \\
			&   \leq 1 - \frac43 s_4 + \frac{4}{9} s_4^2 =   \Big(1-\frac{2}{3}\sum_{i=1}^{n} a_{i}^{4}\Big)^{2}.
		\end{align*}
	\end{proof}
	
	\begin{proof}[Proof of Theorem \ref{thm:cp4}  for $4 \leq p \leq 8$]
		We assume that $\sum_{i=1}^n a_i^2 =1$, in which case our goal reduces to
		\[
		\E|S|^{p} \leq \E|G|^{p} \cdot \left(1-\frac{2}{3}\sum_{i=1}^{n} a_{i}^{4}\right)^{\frac p4}.
		\]
		Using Lemma \ref{lem:KO-1} and its notation, and taking advantage of the fact that
		$(p-1)\,\E|G|^{p-2}=\E|G|^{p}$, we obtain
		\begin{align*}
			\E|S|^{p} & =(p-1)\sum_{i=1}^{n} a_{i}^{2}\,\E|S_{i}|^{p-2}
			\leq (p-1)\sum_{i=1}^{n} a_{i}^{2}\,\E|G|^{p-2}\,\big(\E S_{i}^{2}\big)^{\frac{p-2}{2}} \\
			& =\E|G|^{p}\,\sum_{i=1}^{n} a_{i}^{2}\,\big(\E S_{i}^{2}\big)^{\frac{p-2}{2}}
			=\E|G|^{p}\,\sum_{i=1}^{n} a_{i}^{2}\,\Big(1-\frac{2}{3}a_{i}^{2}\Big)^{\frac{p-2}{2}},
		\end{align*}
		where we have used Haagerup's moment inequality for Rademacher sums; note that $p-2\geq 2$. While $S_{i}$ is not a genuine Rademacher sum,
		since one can express $U$ as $\sum_{j=1}^{\infty} 2^{-j}\tilde{\ve}_j$ , where $(\tilde{\ve}_j)_{j=1}^{\infty}$ is a Rademacher sequence independent of $(\ve_{i})_{i=1}^{\infty}$, the Haagerup moment estimates still apply to $S_{i}$ (passing to limit is trivial here). Thus, 
		\[
		\frac{\E|S|^{p}}{\E|G|^{p}}  \leq \sum_{i=1}^n a_i^2 \left(1 - \frac23 a_i^2\right)^{\frac{p-2}{2}} \le \left(1-\frac23 \sum_{i=1}^n a_i^4\right)^{\frac p4},
		\]
		where the last inequality follows from the fact that its sides are, respectively, log-convex and log-affine functions of $p\in [4,8]$, and thus it is enough to check it only for $p=4,8$, in which case the inequality follows from  Lemma \ref{lem:KO-2} for $p=8$ and holds with equality for $p=4$.
	\end{proof}


	\section{Proof of Theorem \ref{thm:cp4} for $p \geq 5$}\label{sec:p>5}
	
	\noindent In order to deduce Theorem \ref{thm:cp4} from Proposition \ref{prop:x} we need the following lemma.
	
	\begin{lemma}\label{lem:x-gauss}
		Let $p>q>0$. The function
		\[
		\R \ni y \mapsto \frac{\|y+G\|_p}{\|y+G\|_q}
		\]
		is maximized for $y=0$. 
	\end{lemma}
	
	We first show how this lemma implies our main result.
	
	\begin{proof}[Proof of Theorem \ref{thm:cp4}]
		It is enough to show that $C_{p,4} \leq \frac{\gamma_p}{\gamma_4}$, since the reverse inequality holds true, as mentioned earlier. Note that 
		\[
		\lim_{x \to \infty}  \frac{\|x+S_n\|_p}{\|x+S_n\|_4} = \lim_{x \to \infty}  \frac{\|1+\frac{S_n}{x}\|_p}{\|1+\frac{S_n}{x}\|_4} = 1.
		\]
		Since $\|x+S_n\|_p \geq \|x+S_n\|_4$, we get that the function $f_n(x) = \frac{\|x+S_n\|_p}{\|x+S_n\|_4}$ achieves its maximum on $[1,\infty)$ for certain $x_n$. Since by Proposition \ref{prop:x} we have  $C_{p,4,n+1}  = f_n(x_n)$, we see that the sequence $(f_n(x_n))$ is non-decreasing and converges to $C_{p,4}$. By passing to a subsequence we can assume that $\frac{x_n}{\sqrt{n}} \to y \in [0,\infty]$. Let us consider two cases. 
		\vspace{0.2cm}
		
		\noindent \emph{Case 1.} If $y<\infty$ then by homogeneity 
		\[
		f_n(x_n)=  \frac{\|\frac{x_n}{\sqrt{n}}+\frac{S_n}{\sqrt{n}}\|_p}{\|\frac{x_n}{\sqrt{n}}+\frac{S_n}{\sqrt{n}}\|_4} \xrightarrow[n \to \infty]{} \frac{\|y+G\|_p}{\|y+G\|_4}.
		\] 
		Note that passing to the limit is possible due to the fact that if $X_n \to X$ in distribution then $a_n X_n +b_n\to aX+b$  for every sequences $a_n \geq 0$ and $b_n \in \R$ converging respectively to $a$ and $b$. The convergence of $q$th moments (here used with $q=p,4$)  follows from the convergence in distribution since the higher moments are uniformly bounded, e.g.  by Khintchine inequality. Thus by Lemma \ref{lem:x-gauss} we get 
		\[
		C_{p,4} = \lim_{n \to \infty} C_{p,4,n+1} = \lim_{n \to \infty} f_n(x_n) = \frac{\|y+G\|_p}{\|y+G\|_4} \leq \frac{\|G\|_p}{\|G\|_4} = \frac{\gamma_p}{\gamma_4}.
		\]
		
		\noindent \emph{Case 2.} If $y=\infty$ then  by the above mentioned facts, we have $\frac{S_n}{x_n}=  \frac{\sqrt{n}}{x_n} \cdot \frac{S_n}{\sqrt{n}} \to  0$ in distribution and
		\[
		f_n(x_n) = \frac{\|1 + \frac{S_n}{x_n}\|_p}{\|1 + \frac{S_n}{x_n}\|_4} \xrightarrow[n \to \infty]{} 1 \leq \frac{\gamma_p}{\gamma_4}.
		\]
		
	\end{proof}
	
	The proof of Lemma \ref{lem:x-gauss} uses the lemma due to Nazarov and Podkorytov, see \cite{NP00}. In our case, let $\mu$ be a $\sigma$-finite measure on $\R$ and suppose $f:\R \to [0,\infty)$ is measurable. The function 
	\[
	F(t)=\mu(\{x \in \R: \ f(x)>t \}) \qquad (t>0)
	\] 
	is called the distribution function of $f$. For $l>0$ let $\mc{F}_l = \mc F_l(\mu)$ be the space of measurable functions $f:\R \to [0,\infty)$ such that their distribution functions are finite and $\int_\R f^s \dd \mu<\infty$ for all $s>l$. 
	
	\begin{lemma}[Nazarov-Podkorytov]\label{lem:nazarov}
		Let $\mu$ and $\mc F_l$ be as above, and suppose $f,g\in \mc F_l$ have distribution functions $F,G$ such that $F-G$ has one sign-change point $y_0$ and at $y_0$ changes sign from $-$ to $+$. Then
		\[
		\phi(s) = \frac{1}{s y_0^s} \int_\R (f^s-g^s) \dd\mu
		\] 
		is increasing on $(l,\infty)$. In particular, for $s_0>l$, 
		\[
		\int_\R f^{s_0}\dd\mu = \int_\R g^{s_0}\dd\mu  \qquad \implies \qquad \int_\R f^{s} \dd\mu \geq \int_\R g^{s} \dd\mu \qquad \textrm{for all} \ \ s\geq s_0.
		\] 
	\end{lemma}
	
	Let us recall the proof for convenience of the reader. 
	
	\begin{proof}
		Fix $s>l$ and note that applying Fubini theorem to the characteristic function of the set $\{(x,t) \in \R \times (0,\infty) \ : \  f(x)>t^{1/s}\}$, we obtain
		\[
		\int_\R f^s \dd \mu = \int_\R \int_0^\infty\1_{\{f(x)>t^{1/s}\}} \dd t\dd\mu(x) = \int_0^\infty F(t^{1/s}) \dd t = s\int_0^\infty y^{s-1}F(y)dy.
		\]
		Therefore,
		\[
		\phi(s) = \frac{1}{y_0} \int_0^\infty \left( \frac{y}{y_0} \right)^{s-1}(F(y)-G(y)) \dd y .
		\]
		Now, suppose $s_1>s_2>l$ and note that
		\[
		\phi(s_1)-\phi(s_2) = \frac{1}{y_0} \int_0^\infty \left( \left( \frac{y}{y_0} \right)^{s_1-1} - \left( \frac{y}{y_0} \right)^{s_2-1} \right) (F(y)-G(y))  \dd y \geq 0,
		\]
		since both factors change their signs in $y_0$.
	\end{proof}

	\begin{proof}[Proof of Lemma \ref{lem:x-gauss}]
		By symmetry we can assume $y>0$. Let $\gamma$ be a standard Gaussian measure in $\R$, that is, a measure with density $\varphi(s) = (2\pi)^{-1/2} e^{-s^2/2}$.  Our goal is to prove the inequality
		\[
		\|y+G\|_p \leq  C \cdot \|G\|_p, \qquad \textrm{where} \ \ C = \frac{\|y+G\|_q}{\|G\|_q} .
		\]
		This can be rewritten as
		\[
		\int_\R |y+x|^p \dd \gamma(x) \leq  \int_\R |Cx|^p \dd \gamma(x).
		\]
		Note that we have equality for $p=q$. Let us take $f(x)=|Cx|$ and $g(x)=|y+x|$. Due to Lemma \ref{lem:nazarov} it is enough to show that for the corresponding  distribution functions $F,G$ the function $h = F-G$, defined on $(0,\infty)$, changes sign only once and is positive for large arguments. We have
		\[
		h(t) = \gamma(\{x: |y+x| \leq t \}) - \gamma(\{x: |Cx| \leq t \}).
		\]
		Note that the function $y \mapsto \|y+G\|_q$ is even and strictly convex and thus it is minimized for $y=0$. It follows that $\|y+G\|_p > \|G\|_p$ for $y \ne 0$ and thus $C>1$. Since $\{x: |Cx| \leq t\} \subseteq \{x: |y+x| \leq t\}$ for $t \geq t_0 = \frac{Cy}{C-1}$, we have  $h(t) \geq 0$ for $t \geq t_0$. Thus, we are left with showing the sign-change property.
		
		We have $\lim_{t\to 0} h(t)=0$ and $\lim_{t \to \infty} h(t)=0$. Moreover, $\int_\R f^q \dd \gamma= \int_\R g^q \dd\gamma$ implies 
		\[
		0 = \int_0^\infty u^{q-1}(F(u)-G(u)) \dd u= \int_0^\infty u^{q-1} h(u) \dd u
		\]
		and thus $h$ must have at least one sign-change point.  By Rolle's theorem it is enough to show that $h'$ has at most two zeros. Note that 
		\[
		h(t) = \int_{-t-y}^{t-y} \vp(s) \dd s - 2 \int_{0}^{\frac{t}{C}} \vp(s) \dd s  	
		\] 
		and thus
		\[
		h'(t) = \vp(t-y)+\vp(-t-y) - \frac{2}{C} \vp(t/C) = \frac{1}{\sqrt{2\pi}} \left( 2 e^{-\frac12(t^2+y^2)} \cosh(ty) - \frac{2}{C} e^{-\frac{t^2}{2C^2}} \right)
		\]    
		The equation $h'(\sqrt{t})=0$ is therefore equivalent with 
		\[
		\cosh(\sqrt{t}y) = \frac{1}{C}e^{\frac12(t+y^2) - \frac{t}{2C^2}}.
		\]
		By the Hadamard product identity one gets
		\[
		\log \cosh(\sqrt{t}) = \sum_{n=1}^\infty \log(1+a_n t), \qquad \textrm{where} \ \ a_n = \frac{4}{\pi^2(2n-1)^2}
		\]
		and thus the left hand side of the above equation is strictly log-concave in $t$ for $y \ne 0$, whereas the right hand side is log-affine. As a result this equation can have at most two solutions.
	\end{proof}

	\section{Proof of Proposition \ref{prop:x}}\label{sec:prop}
	
	Proposition \ref{prop:x} is almost entirely based on the following lemma. 
	\begin{lemma}\label{lem: max/min on C}
		For an integer $n\ge 2$ and real parameters $\alpha,\beta$, let 
		\[
		A_{\alpha,\beta} := \{ (x_1,\dots,x_n) \in \R^n \ : \ x_1^2+\dots+x_n^2=\alpha^2, \ x_1^4+\dots+x_n^4 = \beta^4\}.
		\]
		Suppose that $A_{\alpha,\beta}$ is non-empty. Then: 	
		\begin{enumerate}
			\item[(a)] there exist unique points $P_\pm \in A_{\alpha,\beta}$ such that
			\[
			P_+= (a_+,b_+,\dots,b_+) \quad \text{and} \quad P_-=(b_-,\dots,b_-,a_-, \underbrace{0,\dots,0 }_{k\text{ times}})  .
			\] 
			for some $a_+\ge b_+\ge 0$, $b_- \ge a_- \ge 0$ and some $k \in \{0,\dots,n-2\}$.
			\item[(b)] For every even function $\Phi:\R \to \R$ with convex fourth derivative, the function
			\[
			A_{\alpha,\beta}\ni (a_1,\dots,a_n) \longmapsto \E \Phi(a_1\ve_1+\dots+a_n\ve_n) 
			\]
			attains the maximal value at $P_+$ and the minimal value at $P_-$. If $\Phi^{(4)}$ is strictly convex, there are no other local extrema, up to coordinate reflections and permutation of coordinates.
		\end{enumerate}  
	\end{lemma}
	
	\begin{remark}
		A small subtlety concerning the definition of $P_-$ is worth mentioning. Though the point $P_-$ is uniquely determined, numbers $b_{-},a_{-}$ and $k$ may not be. In most cases, the latter will also be true, however, if $\frac{\alpha^4}{\beta^4}\in \{2,\dots,n-1\}$, the point $P_-$ will be of the form $(b,\dots,b,0\dots,0)$ and one can take both $b_{-} = a_{-}=b$ and $b_{-}=b, a_{-}=0$. This is essentially the only ambiguity, which of course does not happen in the case of $P_+$.
	\end{remark}
	
	Let us now deduce Proposition \ref{prop:x} from Lemma \ref{lem: max/min on C}.
	
	\begin{proof}[Proof of Proposition \ref{prop:x}, assuming Lemma \ref{lem: max/min on C}]
		As we mentioned, $C_{p,4,n+1}$ is attained for some $S = \sum_{i=0}^n a_i\ve_i$, thus it is enough to prove that
		\[
		\frac{\|S\|_p}{\|S\|_4} \le \sup_{x\ge 1} \frac{\|x+S_n\|_p}{\|x+S_n\|_4}
		\]
		Notice that for $p\ge 5$ the function $\Phi(x) = |x|^p$ is even and $\Phi^{(4)}(x) = p(p-1)(p-2)(p-3)|x|^{p-4}$ is convex. Applying Lemma \ref{lem: max/min on C}, there exist $a\ge b\ge 0$ such that
		\[
		\E |S|^p \le \E |a \ve_0 +bS_n|^p \qquad \text{ and } \quad \begin{cases}
			a_0^2+\dots+a_n^2 = a^2 +nb^2 \\
			a_0^4+\dots +a_n^4 = a^4 + nb^4
		\end{cases} 
		\]
		Thus, $\|S\|_p \le \|a\ve_0+bS_n\|_p$ and by the formula for the fourth moment of a Rademacher sum we also have $\|S\|_4 = \|a\ve_0 + bS_n\|_4$. Therefore, for $b\not =0$, since $\frac ab \ge 1$, we obtain
		\[
		\frac{\|S\|_p}{\|S\|_4}\le \frac{\|a\ve_0 + bS_n\|_p}{\|a\ve_0 + bS_n\|_4} = \frac{\|\frac ab \ve_0 +S_n\|_p}{\|\frac ab \ve_0 +S_n\|_4} \le \sup_{x\ge 1} \frac{\|x+S_n\|_p}{\|x+S_n\|_4}
		\]
		If $b=0$, then $\frac{\|S\|_p}{\|S\|_4}\le1$, which gives the required inequality due to the monotonicity of norms.
	\end{proof}
	
	\subsection{The set $A_\gamma$ and its special points}\label{sub: set A_gamma} We will consider the following family of sets
	\begin{equation}\label{eq:C2}
		A_\gamma := \{(x,y,z) \in \R^3 \ : \ x^2+y^2+z^2=1, \ x^4+y^4+z^4=\gamma \}
	\end{equation}
	for $\frac13\le \gamma\le 1 $. Since for any point $(x,y,z)$ on the unit sphere $S^2\subseteq \R^3$, we have 
	\begin{equation}\label{ineq x,y,z} \small
		\frac13 =\frac{(x^2+y^2+z^2)^2}{3}\le x^4+y^4+z^4 \le (x^2+y^2+z^2)^2=1,
	\end{equation} 
	this is simply a partition of $S^2$ into level-sets of $S^2\ni (x,y,z) \mapsto x^4+y^4+z^4$. Below we include the pictures of $A_\gamma$, depending on the value of $\gamma \in (\frac13,1)$. 
	
	\begin{figure}[h]
		\centering
		\includegraphics[scale=0.35]{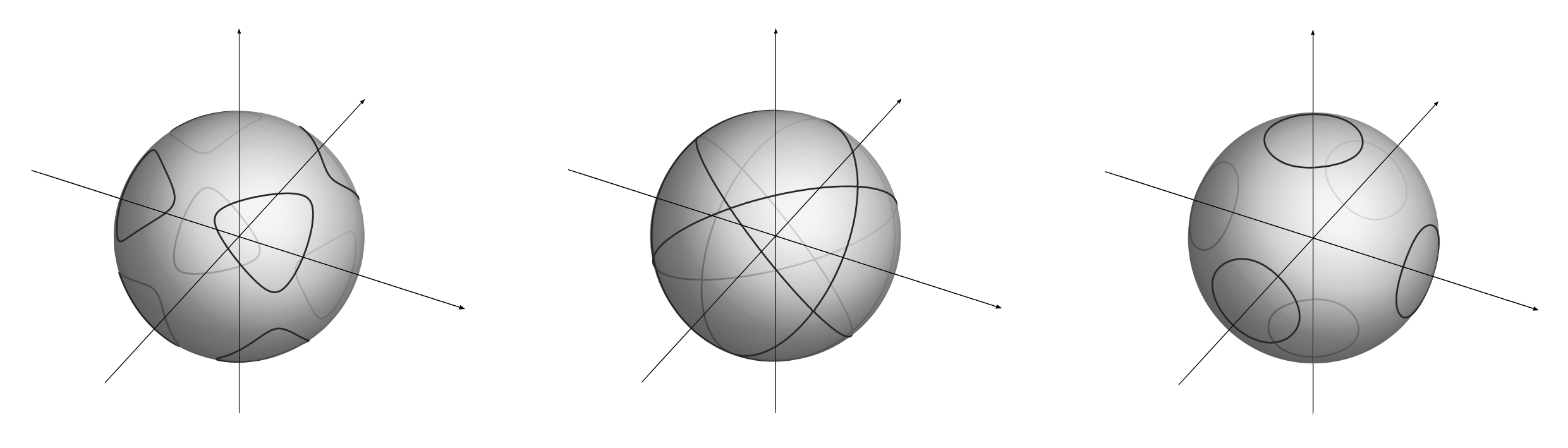}
		\caption{The sets $A_\gamma$ for $\gamma \in(\frac13,\frac12)$, $\gamma=\frac12$, and $\gamma \in (\frac12,1)$.}
		\label{fig:kule}
	\end{figure}
	
	We can see that $A_\gamma$ are invariant under permuting coordinates and coordinate reflections, i.e. invariant under action of the full octahedral group. We shall call these transformations \emph{symmetries} and consider set of points on $S^2$ with non-trivial stabilizer, called \emph{special points} and denoted by $\mc S$ in the sequel. The fundamental domain is given by $\mc F =\{(x,y,z) \in S^2 \ : \ x\ge y\ge z\ge 0\}$ and the representatives of special points are the boundary points. On this boundary, there are exactly two (representative) paths connecting the points $(1,0,0)$ and $(\frac1{\sqrt3}, \frac1{\sqrt3}, \frac1{\sqrt3})$, namely 
	\[
	\mc S^{-} = \mc F \cap (\{z=0\} \cup \{x=y\}), \qquad   \mc S^{+} = \mc F \cap \{y=z\} .
	\]
	Note that $(\frac{1}{\sqrt2}, \frac1{\sqrt2}, 0) \in  \mc S^{-}$.	We will also consider special points on $A_\gamma$ for $\gamma \in [\frac13,1]$, namely
	\[
	\mc S_\gamma = A_\gamma \cap \mc{S}, \qquad \qquad \mc S_\gamma^\pm = A_\gamma \cap \mc{S}^\pm.
	\]
	The characterization of these special points is gathered in the following elementary lemma. 
	
	\begin{lemma}\label{lem: special points}
		Let $A_\gamma$ be the set defined in \eqref{eq:C2} with $\gamma \in [\frac13,1]$, and let $\mc S_\gamma,\mc S_\gamma^\pm$ denote the (sub)sets of special points on $A_\gamma$, as defined above. Then
		\begin{enumerate}
			\item[(a)]  $(x,y,z) \in S^2$ is not a special point if and only if $|x|,|y|,|z|$ are non-zero and pairwise distinct.
			\item[(b)] for $\gamma =\frac13$ we have $A_\gamma = \mc S_\gamma = \mc S_\gamma^\pm = \{(\frac1{\sqrt3}, \frac1{\sqrt3}, \frac1{\sqrt3})\}$, up to symmetries, 
			\item[(c)]  for $\gamma = 1$ we have $A_\gamma = \mc S_\gamma = \mc S_\gamma^\pm = \{(1,0,0)\}$, up to symmetries, 
			\item[(d)] for $\gamma \in (\frac13,1)$ the set $\mc S_\gamma$ contains, up to symmetries, exactly one point in $\mc S_\gamma^+$ and one point in $\mc S_\gamma^-$. The representatives of special points are:
			{
				\[
				\mc S_\gamma^+	\ni s_\gamma^+  := \left( \sqrt{\frac{1+\sqrt{6\gamma-2}}{3}}, \sqrt{\frac{2-\sqrt{6\gamma-2}}{6}},\sqrt{\frac{2-\sqrt{6\gamma-2}}{6}} \right)  \quad \text{ for } \gamma \in \left(\frac13,1\right)
				\]
			}
			\[
			\mc S_\gamma^-	\ni s_\gamma^-  := \begin{cases}
				\left( \sqrt{\frac{2+\sqrt{6\gamma-2}}{6}},\sqrt{\frac{2+\sqrt{6\gamma-2}}{6}}, \sqrt{\frac{1-\sqrt{6\gamma-2}}{3}} \right) & \text{ for } \gamma \in \left(\frac13,\frac12\right] \\
				\left( \sqrt{\frac{1+\sqrt{2\gamma-1}}{2}},\sqrt{\frac{1-\sqrt{2\gamma-1}}{2}},0 \right) &\text{ for } \gamma \in \left(\frac12,1\right)
			\end{cases}
			\]
			In particular,  $s_{1/2}^+ = \left( \sqrt{\frac23}, \frac{1}{\sqrt 6}, \frac1{\sqrt 6}\right)$ and $s_{1/2}^- = \left(\frac1{\sqrt2}, \frac1{\sqrt2}, 0\right)$.
		\end{enumerate}	
	\end{lemma}
	
	\begin{proof}
		Point (a) follows immediately from the fact that points with a zero coordinate are exactly the points invariant under certain coordinate reflection and the points with two equal coordinates  are points invariant under certain permutation.  
		
		Points (b) and (c) follow from the equality conditions in \eqref{ineq x,y,z}.
		
		For (d) let us consider the system of equations, corresponding to  special points with two equal coordinates, 
		\[
		\begin{cases}
			a^2+2b^2 = 1 \\
			a^4+2b^4 = \gamma,
		\end{cases}   \qquad a,b \geq 0.
		\]
		Here $(a^2,b^2)=(\frac{1+\sqrt{6\gamma-2}}{3},\frac{2-\sqrt{6\gamma-2}}{6})$ is always the solution with non-negative coordinates satisfying $a^2 \geq b^2$ for all $\gamma \in(\frac13,1)$, which leads to the point $s_\gamma^+$. The second solution $(a^2,b^2)=(\frac{1-\sqrt{6\gamma-2}}{3},\frac{2+\sqrt{6\gamma-2}}{6})$ has real positive coordinates only if $\gamma \in (\frac13,\frac12]$ and in this case $b^2 \geq a^2$, which leads to the point $s_\gamma^-$ in this range of $\gamma$. The second type of special points are the points of the form $(a,b,0)$ with $a \geq b$, which corresponds to the system of equations
		\[
		\begin{cases}
			a^2+b^2 = 1 \\
			a^4+b^4 = \gamma,
		\end{cases}   \qquad a^2 \geq b^2 \geq 0.
		\]
		The solution exists only for $\gamma \in [\frac12,1]$ and is equal to  $(a^2,b^2)=(\frac{1+\sqrt{2\gamma-1}}{2},\frac{1-\sqrt{2\gamma-1}}{2})$, which leads to $s_\gamma^-$ for $\gamma \in (\frac12,1)$. \end{proof}

	One can easily verify that $M_\gamma := A_\gamma \setminus \mc S_\gamma$ defines a differentiable manifold for $\gamma \in (\frac13,1)$. Nevertheless, we will prove this fact along the way at the end of Section \ref{sub: new variables}.
	
	\subsection{New system of variables}\label{sub: new variables}
	In this section we fix $\gamma \in (\frac13,1)$. Let us introduce the following change of variables,   
	\[
	\Lambda \ : \quad \R^3 \ni \left( \begin{array}{c}
		x \\ y \\ z
	\end{array}\right) \stackrel{\Lambda}{\longmapsto} \left( \begin{array}{rrr}
		1 & 1 & 1 \\ 1 & -1 & -1 \\ -1 & 1 & -1 \\ -1 & -1 & 1
	\end{array} \right) \left( \begin{array}{c}
		x \\ y \\ z
	\end{array}\right)  =\left( \begin{array}{c}
		a \\ b \\ c \\ d
	\end{array}\right)  \in \R^4 .
	\]
	The map $\Lambda$ is a diffeomorphism between $\R^3$ and the hyperplane \[
	H:= \{(a,b,c,d)\in \R^4\  : \ a+b+c+d=0\}
	\]
	with an inverse 
	\[
	\Lambda^{-1}(a,b,c,d) = \left(-\frac{c+d}{2}, -\frac{d+b}{2}, -\frac{b+c}{2}\right),
	\]
	which yields bijections of sets,
	\[
	M_\gamma \longleftrightarrow \Lambda(M_\gamma) =: N_\gamma, \quad \mc S_\gamma \longleftrightarrow   \Lambda (\mc{S}_\gamma) =: \mc P_\gamma, \quad \text{and} \quad \mc S \longleftrightarrow  \Lambda(\mc{S}) =: \mc P .
	\]
	Notice that the symmetries in $\R^3$ correspond to automorphisms of $H$, which will be called \emph{induced symmetries}. Since the symmetries can be represented as generalized $3\times 3$ permutation matrices with $\pm1$ entries, by checking the image of a symmetry $\sigma$ under $\sigma \mapsto \Lambda \sigma \Lambda^{-1}$, we can see that the induced symmetries are compositions of negation of a vector and permutations of its coordinates in $\R^4$ (restricted to $H$). By \textit{induced special points} we define the points on $H$ with non-trival stabilizer under action of the induced symmetries. Clearly  $\mc P_\gamma$ is precisely the set of induced special points on $\Lambda(A_\gamma)$.
	
	Now, since for any $(x,y,z)\in \R^3$ and $(a,b,c,d) = \Lambda(x,y,z)$, we have 
	{ \[
		\begin{cases}
			x^2+y^2+z^2 = \frac14(a^2+b^2+c^2+d^2)\\
			x^4+y^4+z^4 = \frac12 \left( \left(\frac{a^2+b^2+c^2+d^2}{4}\right)^2+abcd \right)
		\end{cases},
		\]}considering the constraints on $A_\gamma$ composed with the map $\Lambda^{-1}$ on $H$ yields
	\[
	N_\gamma = \{(a,b,c,d) \in \R^4 \ : \ a+b+c+d=0,\ {a^2+b^2+c^2+d^2}= 4,\ abcd = 2\gamma-1\}\setminus \mc P_\gamma.
	\]
	Therefore, $N_\gamma = F^{-1}(0,4,2\gamma-1)$ for $F : U:=\R^4\setminus \mc P \to \R^3$, where 
	\begin{equation}\label{eq: def of F}
		F(q) =\left(\begin{matrix}
			F_1\\F_2\\F_3
		\end{matrix}\right)(q):= \left(\begin{matrix}
			a+b+c+d\\ a^2+b^2+c^2+d^2 \\ abcd
		\end{matrix}\right) \quad \text{ for } q=(a,b,c,d)\in U.
	\end{equation}
	Notice that $U$ is open in $\R^4$ (since the set of induced special points is clearly closed), and 
	\begin{equation}\label{eq: matrix DF}
		(DF)(q) = \left( \begin{array}{cccc}
			1 & 1 & 1 & 1 \\
			2a & 2b & 2c & 2d \\
			bcd & cda & dab & abc 
		\end{array} \right) , 
	\end{equation}
	is of full rank for any $q\in U$, by the following lemma. 
	\begin{lemma}\label{lem: ind special points}
		Suppose that $q=(a,b,c,d)\in H$ is not an induced special point. Then, 
		\begin{enumerate}
			\item[(a)] the set $\{\pm a,\pm b, \pm c, \pm d\}$ contains at least 7 distinct elements.
			\item[(b)] the matrix $DF(q)$ defined in \eqref{eq: matrix DF} has rank three. 
		\end{enumerate}
	\end{lemma}
	
	\begin{proof}
		(a) It suffices to show that $|a|,|b|,|c|,|d|$ are pairwise distinct. If, say, $a = b$, then $(a,b,c,d) = (b,a,c,d)$ and if $a = -b$ then $(a,b,c,d)=-(b,a,d,c)$ and thus in both cases   $(a,b,c,d) \in \mc P$.
		
		(b) If $abcd =0$, then we can assume $a=0$, since other cases are similar. Consequently,
		{\small \[
			(DF)(p) = \left( \begin{array}{cccc}
				1 & 1 & 1 & 1 \\
				0 & 2b & 2c & 2d \\
				bcd & 0 & 0 & 0 
			\end{array} \right) .
			\]}This matrix has rank three as $0,b,c,d$ are pairwise distinct by (a). To see this compute the determinant of $3\times 3$ matrix obtained by deleting the last column.
		
		For $abcd\not =0$, note that dividing a row or column by a non-zero number does not change the rank. Thus,  
		{\small \[
			\rk \left( \begin{array}{cccc}
				1 & 1 & 1 & 1 \\
				2a & 2b & 2c & 2d \\
				bcd & cda & dab & abc 
			\end{array} \right) =\rk 
			\left( \begin{array}{cccc}
				1 & 1 & 1 & 1 \\
				a & b & c & d \\
				a^{-1} & b^{-1} & c^{-1} & d^{-1} 
			\end{array} \right) =\rk 
			\left(\begin{array}{cccc}
				a & b & c & d \\
				a^2 & b^2 & c^2 & d^2 \\
				1 & 1 & 1 & 1 
			\end{array} \right).
			\]}Note that any $3\times 3$ submatrix of the last matrix is of Vandermonde type, and thus has rank three as $a,b,c,d$ are distinct by (a).
	\end{proof}
	Therefore, $N_\gamma$ is a differentiable manifold, and since $\Lambda: \R^3 \to H$ is a diffeomorphism with $\Lambda(M_\gamma) =  N_\gamma$, this implies that $M_\gamma$ is in fact a differentiable manifold diffeomorphic to $N_\gamma$.
	
	\subsection{The function $\Phi$ and its properties}\label{sub: functions properties} Let us introduce the following lemmas needed in the proof of Lemma \ref{lem: max/min on C}.

	\begin{lemma}\label{lem: 3 solutions}
		Suppose that $\Phi :\R \to \R$ is an even function with strictly convex fourth derivative. Then for every $\alpha \in \R$, the equation
		\begin{equation}\label{3 solutions}
			(u\Phi'(u))^{(3)} = \alpha u,
		\end{equation}
		has at most three real solutions.
	\end{lemma}
	
	\begin{proof}
		Since both $ (u\Phi'(u))^{(3)} $ and $\alpha u$ are odd functions, it is enough to prove that there is at most one positive solution. Notice that for $u>0$, equation \eqref{3 solutions} is equivalent to
		{\small \[
			\Phi^{(4)}(u) + 3\frac{\Phi^{(3)}(u)}{u} = \alpha,
			\]}and it suffices to prove the monotonicity of the left side. Since $\Phi^{(4)}$ is even and strictly convex, it follows that $\Phi^{(4)}$ is strictly increasing on $[0,\infty)$. Thus, $\Phi^{(3)}$ is (strictly) convex on $[0,\infty)$, which implies the monotonicity of slopes $\frac{\Phi^{(3)}(u) - \Phi^{(3)}(0)}{u} = \frac{\Phi^{(3)}(u)}{u}$ for $u>0$. Therefore, both $\Phi^{(4)}(u)$ and $\frac{\Phi^{(3)}(u)}{u}$ are strictly increasing on $(0,\infty)$, which ends our proof. 
	\end{proof}

	\begin{lemma}\label{lem: f(sqrt)}
		Suppose that $\Phi : \R \to \R$ is an even function with strictly convex fourth derivative. For any $a>0$, function 
		\[
		(0,a] \ni r\longmapsto \frac{\Phi(\sqrt{a+r}) -\Phi(\sqrt{a-r})}{2r}
		\] 
		is strictly increasing. 
	\end{lemma}
	
	\begin{proof}
		Define $\varphi(u) = \Phi(\sqrt u)$ for $u> 0$, and notice that 
		{\small \[
			\frac{\Phi(\sqrt{a+r}) -\Phi(\sqrt{a-r})}{2r} = \frac{\varphi(a+r) - \varphi(a-r)}{2r}= \int_0^1 \frac{\varphi'(a+tr) +\varphi'(a-tr)}{2}dt, 
			\]}thus, it is enough to prove the monotonicity of $s \mapsto \varphi'(a+s) +\varphi'(a-s)$ on $(0,a]$. Equivalently, $\varphi''(a+s) > \varphi''(a-s)$ for $s\in (0,a)$, and it suffices to show that $\varphi''$ is strictly increasing on $(0,\infty)$. Since $\varphi'''(u) = \frac{ u \Phi^{(3)}(\sqrt{u}) -3\sqrt{u} \Phi''(\sqrt{u}) +3\Phi'(\sqrt{u})}{8u^{5/2}}$, we need $\psi(t) = t^2\Phi^{(3)}(t) - 3t\Phi''(t)+3\Phi'(t)>0$ for $t>0$. Notice that $\psi(0)=0$ and $\psi'(t) =  t^2\left(\Phi^{(4)}(t) - \frac{\Phi^{(3)}(t)}{t}\right) = t^2(\Phi^{(4)}(t) - \Phi^{(4)}(t_0))$ for some $t_0 \in (0,t)$ by the mean value theorem, and since $\Phi^{(4)}$ is strictly increasing on $(0,\infty)$ as an even strictly convex function, we have $\psi'>0$ on $(0,\infty)$. Thus, $\psi>0$ on $(0,\infty)$, which ends the proof. 
	\end{proof}
	
	\begin{remark}
		In the proof of Lemma \ref{lem: max/min on C}, these results will be used for the family of functions $\Phi_s(u): = \frac12\left(\Phi(u+s)+\Phi(u-s)\right)$ where $s\in \R$. One can verify that for an even function $\Phi \in \mathcal{C}^4$, the result of Lemma \ref{lem: 3 solutions} holds for all $\Phi_s$ ($s\in \R$) if and only if $\Phi^{(4)}$ is strictly convex.
	\end{remark}
	
	\subsection{Proof of Lemma \ref{lem: max/min on C}}\label{sub: proof of lemma} We now have enough tools to prove Lemma \ref{lem: max/min on C}. In the following section we will consistently use the formulation of Lemma \ref{lem: max/min on C} along with notations from Sections \ref{sub: set A_gamma}, \ref{sub: new variables}. 
	
	\begin{proof}[Proof of Lemma \ref{lem: max/min on C}] Throughout the proof we will assume $n\ge3$ and $\alpha,\beta\not =0$, since otherwise, the claims (a) and (b) follow instantly. 
		
		We will begin with the claim (a), i.e. the existence and uniqueness of $P_\pm$. The set $A_{\alpha,\beta}$ is non-empty if and only if $\beta^4 \leq \alpha^4 \leq n \beta^4$. Let us consider the system of equations
		\[
		\begin{cases}
			a^2+lb^2 = \alpha^2 \\
			a^4+lb^4 = \beta^4.
		\end{cases}
		\]
		When $l=n-1$ this system has a unique solution $(a_+^2, b_+^2)$ satisfying $a_+^2 \geq b_+^2 \geq 0$, namely
		\[
		a_+^2 = \frac{\alpha^2}{n} + (n-1) \sqrt{\frac{\beta^4}{(n-1)n}-\frac{\alpha^4}{(n-1)n^2}}, \qquad b_+^2 =  \frac{\alpha^2}{n} -  \sqrt{\frac{\beta^4}{(n-1)n}-\frac{\alpha^4}{(n-1)n^2}},
		\]
		which leads to the point $P_+$. For $l \in \{1,\dots,n-1\}$ the solution satisfying $b_l^2 \geq a_l^2 \geq 0$ exists  if and only if $l+1\ge \frac{\alpha^4}{\beta^4}\ge l$ and is given by
		\[
		a_l^2 = \frac{\alpha^2}{l+1} - l \sqrt{\frac{\beta^4}{l(l+1)}-\frac{\alpha^4}{l(l+1)^2}}, \qquad b_l^2 =  \frac{\alpha^2}{l+1} +  \sqrt{\frac{\beta^4}{l(l+1)}-\frac{\alpha^4}{l(l+1)^2}}.
		\] 
		This proves the existence and uniqueness of $P_-$, provided that $\frac{\alpha^4}{\beta^4}\notin \Z$. For $\frac{\alpha^4}{\beta^4}\in \Z$ one has $a_l^2=b_l^2=\frac{\beta^4}{\alpha^2}$  when $l=\frac{\alpha^4}{\beta^4}-1$, whereas $b_l^2=\frac{\beta^4}{\alpha^2}$  and $a_l=0$ for $l=\frac{\alpha^4}{\beta^4}$, which gives existence and uniqueness of $P_-$.

		Let us move on to the proof of (b). Firstly, we  prove the claim for all even functions $\Phi$ with strictly convex fourth derivative. The proof is by induction, and the main part concerns the base case.
		
		Suppose that $n=3$. Notice that we can fix $\alpha^2=1$ since $u\mapsto \Phi(\alpha u)$ is an even function with strictly convex fourth derivative. Then $A_{\alpha,\beta} = A_\gamma$ with $\gamma \in [\frac13,1]$ by the non-emptiness of $A_{\alpha,\beta}$, and the cases $\gamma\in \{\frac13,1\}$ follow directly from Lemma \ref{lem: special points}(b,c). Therefore we assume $\gamma\in (\frac13,1)$, and consider
		\[
		\Theta(x,y,z) = \E \Phi(x\ve_1+y\ve_2+z\ve_3) \qquad \text{on } A_\gamma = M_\gamma \cup \mc S_\gamma.
		\]
		This is a $C^1$ function on $\R^3$, and a diffeomorphism $M_\gamma \simeq_\Lambda N_\gamma$ yields a bijection of sets
		\[
		\{\text{critical points of } \Theta \text{ on } M_\gamma\} \longleftrightarrow \{\text{critical points of } \Theta\circ \Lambda^{-1} \text{ on } N_\gamma\}.
		\] 
		We will prove that these sets are empty. Suppose otherwise, and notice that $\Psi:\R^4 \to \R$ given by 
		\[
		\Psi(a,b,c,d) = \frac14 \left( \Phi(a)+\Phi(b)+\Phi(c)+\Phi(d) \right)
		\]
		is a $C^1$ function on $\R^4$, satisfying $\Psi = \Theta \circ \Lambda^{-1}$ on $H$. Therefore, using Lagrange multipliers theorem for $\Psi $ on $N_\gamma$ (which can be used due to the definition of $F$ in \eqref{eq: def of F} and Lemma \ref{lem: ind special points}(b)), any critical point $q=(a,b,c,d)$ of $\Theta\circ \Lambda^{-1}$ on $N_\gamma$ satisfies the system of equations
		\[
		(\nabla \Psi)(q) = \lambda_1 \nabla F_1 (q) +\lambda_2 \nabla F_2(q) + \lambda_3 \nabla F_2(q),
		\]
		for some $\lambda_1,\lambda_2,\lambda_3\in \R$. Multiplying equations respectively by $4a,4b,4c,4d$, implies that
		\begin{equation}\label{eq: malo rozw}
			u\Phi'(u) = \tau_1 + \tau_2 u + \tau_3 u^2 \qquad \text{ for } u=a,b,c,d,
		\end{equation}
		where $\tau_1 = 4\lambda_3(2\gamma-1)$, $\tau_2= 4\lambda_1$ and $\tau_3 = 8\lambda_2$. Define
		\[
		P(u) = (u-a)(u-b)(u-c)(u-d) = u^4 +e_3u^3+e_2u^2+e_1u+e_0.
		\]
		and note that by Vieta's formulas
		\[
		\begin{cases}
			e_3= -(a+b+c+d)=0 \\
			e_1 = -(bcd+cda+dab+abc) = -8xyz
		\end{cases},
		\] 
		where $(x,y,z) = \Lambda^{-1}(q)$. Since $q\notin \mc P_\gamma$, we have $\Lambda^{-1}(q) \notin \mc S_\gamma$, consequently $e_1\not =0$ by Lemma \ref{lem: special points}(a). Thus, 
		\[
		\tau_1 +\tau_2 u+\tau_3 u^2- \frac{\tau_2}{e_1}P(u) = \eta_1 +\eta_2 u^2 + \eta_3 u^4
		\]
		for some $\eta_1,\eta_2,\eta_3 \in \R$, and any solution of \eqref{eq: malo rozw} yields
		\[
		u\Phi'(u) = \eta_1 +\eta_2 u^2 + \eta_3 u^4 \qquad \text{ for } u=a,b,c,d.
		\]
		Since both sides are even functions of $u$, the last equation is satisfied for any $u = \{\pm a,\pm b,\pm c,\pm d\}$, thus has at least 7 distinct solutions by Lemma \ref{lem: ind special points}(a). Therefore, by Rolle's theorem, equation 
		\[
		\left(u\Phi'(u) \right)^{(3)}= 24\eta_3 u,
		\]
		has at least four distinct solutions, which is in contradiction with Lemma \ref{lem: 3 solutions}.
		
		Therefore, there are no critical points of $\Theta$ on $M_\gamma$, and by compactness of $A_\gamma$ the function $\Theta$ attains extremal values in $\partial M_\gamma=\mc S_\gamma$. By Lemma \ref{lem: special points}(d) in $\mc S_\gamma$ there are, up to symmetries, exactly two points, $s_\gamma^+$ and $s_\gamma^-$. Since $\Theta$ is  invariant under symmetries,  it is enough  to prove that
		\[
		\Theta(s_\gamma^-) < \Theta(s_\gamma^+),
		\]  
		which would imply that the maximal and minimal values are attained only in $\mc S_\gamma^+$ and $\mc S_\gamma^-$, respectively. Note that we cannot have an equality for any $\gamma \in (\frac13,1)$, since this would imply that $\Theta$ is constant on $A_\gamma$, and every point on $M_\gamma$ would be critical. Moreover, using continuity of $(\frac13,1)\ni \gamma \mapsto s_\gamma^{\pm} $ together with the intermediate value property, it is enough to check the inequality for a single $\gamma\in (\frac13,1)$. Taking $\gamma=\frac12$, and using formulas for $s_{1/2}^\pm$ from Lemma \ref{lem: special points}(d), we can see that
		\[
		\Theta(s_{1/2}^-) < \Theta(s_{1/2}^+)
		\] 
		is equivalent to $\frac{\Phi(\sqrt2) - \Phi(\sqrt{\frac23})}{4/3}< \frac{\Phi(\sqrt{\frac83}) - \Phi(0)}{8/3} $, which holds due to the Lemma \ref{lem: f(sqrt)} for $a = \frac43$, $r=\frac23,\frac43$.
		
		We now move on to the inductive step. Fix $n\ge 3$ and suppose that the claim (b) holds for all even $\Phi$ with strictly convex $\Phi^{(4)}$. We will prove that the same is true for $n+1$. By compactness of $A_{\alpha,\beta}$ in $\R^{n+1}$ the function $\E\Phi(a_0\ve_0+\dots+a_n\ve_n)$ attains maximal and minimal value at some $a^* = (a_0^*, \dots,a_n^*)\in \R^{n+1}$ and $a' = (a_0', \dots,a_n')\in\R^{n+1}$, respectively. Assuming, $a_0^*\ge \dots\ge a_n^*\ge 0$ and $a_0'\ge \dots a_n'\ge 0$, it is enough to prove that $a^* = P_+$ and $a' =P_-$. For fixed $s\in \R$ consider
		\[
		\Phi_s(u) := \E_{\ve}\Phi(u+s\ve) = \frac12 \left(\Phi(u+s)+\Phi(u-s)\right).
		\]
		The functions $\Phi_s$ are even with strictly convex fourth derivative, so applying induction hypothesis to $\Phi_{s}$ on $\{(x_1,\dots,x_n)\in \R^n : \sum_{i=1}^nx_i^2 = \alpha^2-s^2, \ \sum_{i=1}^n x_i^4=\beta^4-s^4\}$ for $s=a_{0}^*$ and $s=a_n^*$ gives $a_0^*\ge a_1^*=\dots=a_n^*$ by the form of maximizers in dimension $n$. Thus, by uniqueness of $P_+$, we obtain $a^*=P_+$. Similarly, $s=a_0',a_{n}'$ yield $a'=P_-$, ending this part.   
		
		To conclude the proof for all even functions $\Phi$ with convex fourth derivative, notice that every such function is a pointwise limit of a sequence of even functions with strictly convex fourth derivative, for which we know that the extrema of the corresponding functions are attained at $P_\pm$. Since these functions converge pointwise as well, passing to the limit for any fixed $a\in A_{\alpha,\beta}$ yields the claim.
	\end{proof}

	\section{Proof of Proposition \ref{prop:Sn_to_G}}\label{sec:opt}
	
	We shall use the following lemma.
	\begin{lemma}\label{lem:Sn_to_S2n}
		For any positive integer $n$ and $p \geq 3$ we have 
		$$\Big\|\frac{S_{2n}}{\sqrt{2n}}\Big\|_p \leq e^{p/4n} \Big\|\frac{S_n}{\sqrt{n}}\Big\|_p.$$
	\end{lemma}
	\begin{proof}
		Recall that $S_{2n} = \sum_{i = 1}^{2n} \ve_i$ and let $\rho_i = \ve_{2i-1} + \ve_{2i}$ for $i = 1, 2, \dots, n$. Then, $S_{2n} = \sum_{i = 1}^n \rho_i$. Let $X$ be the number of $i \in \{0,\dots,n\}$ such that $\rho_i \neq 0$. Notice that $S_{2n} = 2S_X$ (we set $S_0 = 0$) and the sequence $(\E|\frac{S_k}{\sqrt k}|^p)_{k=1}^n$ is non-decreasing for $p\ge 3$, thus
		\[
		\E \left|\frac{S_{2n}}{\sqrt{2n}}\right|^p =  \E\left|\frac{2S_X}{\sqrt{2n}}\right|^p = \left(\frac 2n \right)^{\frac p2}\E \left(X^{\frac p2} \E\left( \left|\frac{S_X}{\sqrt{X}}\right|^p \ \Big| \ X \in \{1,\dots,n\}\right)\right) \le \left(\frac 2n \right)^{\frac p2} \E X^{\frac p2} \E \left|\frac{S_n}{\sqrt n} \right|^p.
		\]
		By Corollary 1 of \cite{TA21} we have
		$$\E |X|^{p/2} \leq \Big(\frac{n}{2}\Big)^{p/2} e^{p^2/4n}.$$
		Thus,
		\[
		\E \Big|\frac{S_{2n}}{\sqrt{2n}}\Big|^p \leq e^{p^2/4n} \E \Big|\frac{S_n}{\sqrt{n}}\Big|^p.
		\]
	\end{proof}
	\noindent We are now ready to prove the proposition.
	
	\begin{proof}[Proof of Proposition \ref{prop:Sn_to_G}]
		By Lemma \ref{lem:Sn_to_S2n} we know that for $k = 0, 1, 2, \dots$ we have
		\[
		\Big\|\frac{S_{2^{k+1}n}}{\sqrt{2^{k+1}n}}\Big\|_p \ / \  \Big\|\frac{S_{2^{k}n}}{\sqrt{2^kn}}\Big\|_p \leq \exp\Big(2^{-k-2}p/n\Big).
		\]
		Multiplying those inequalities for $k = 0, 1, \dots, m-1$ we get 
		$$\Big\|\frac{S_{2^{m}n}}{\sqrt{2^mn}}\Big\|_p \leq  \exp\Big((1-2^{-m})\frac{p}{2n}\Big)\Big\|\frac{S_n}{\sqrt{n}}\Big\|_p.$$
		We finish the proof by letting $m \to \infty$.
	\end{proof}

	\end{document}